\documentclass[letterpaper, 10 pt, conference]{ieeeconf}  

\IEEEoverridecommandlockouts                   
\usepackage{cite}
\usepackage{amsmath,amssymb,amsfonts}

\usepackage{graphicx}
\graphicspath{ {./images/} }
\usepackage{comment}
\usepackage{enumerate}
\usepackage{mathtools}

\DeclareMathOperator{\conv}{co}
\DeclareMathOperator{\bd}{bd}

\newcommand{\Ga}{\mathcal{G}_\alpha}
\DeclareMathOperator{\sign}{sign}

\newtheorem{theorem}{Theorem}

\newtheorem{proposition}{Proposition}
\newtheorem{definition}{Definition}
\newtheorem{assumption}{Assumption}

\usepackage{etoolbox}

\BeforeBeginEnvironment{theorem}{\vspace{2mm}}
\AfterEndEnvironment{theorem}{\vspace{2mm}}

\BeforeBeginEnvironment{lemma}{\vspace{2mm}}
\AfterEndEnvironment{lemma}{\vspace{2mm}}

\BeforeBeginEnvironment{remark}{\vspace{2mm}}
\AfterEndEnvironment{remark}{\vspace{2mm}}

\BeforeBeginEnvironment{proposition}{\vspace{2mm}}
\AfterEndEnvironment{proposition}{\vspace{2mm}}

\BeforeBeginEnvironment{definition}{\vspace{2mm}}
\AfterEndEnvironment{definition}{\vspace{2mm}}

\BeforeBeginEnvironment{assumption}{\vspace{2mm}}
\AfterEndEnvironment{assumption}{\vspace{2mm}}

\BeforeBeginEnvironment{corollary}{\vspace{2mm}}
\AfterEndEnvironment{corollary}{\vspace{2mm}}
\newcommand{\omegaf}{\omega_{\textup{f}}}
\newcommand{\thetar}{\theta_{\textup{r}}}

\title{\LARGE \bf
Generalized Multi-Constraint Extremum Seeking
}

\author{Alan Williams$^{1}$ \quad  Jorge Cortés$^{2}$ \quad Alexander Scheinker$^{3}$
\thanks{This work was supported by the U.S. Department of Energy (DOE), Office of Science, Office of High Energy Physics contract number 89233218CNA000001 and the Los Alamos National Laboratory LDRD Program Directed Research (DR) project 20220074DR}
\thanks{$^{1}$Alan Williams and $^{3}$Alexander Scheinker are with the Accelerator Operations and Technology - Instrumentation and Control (AOT-IC) group at Los Alamos National Lab, Los Alamos, NM 87545, USA (e-mail: \{awilliams, ascheink\}@lanl.gov).}
\thanks{$^{2}$Jorge Cortés is with the Department of Mechanical and Aerospace Engineering, University of California, San Diego, CA 92093-0411, USA (e-mail: cortes@ucsd.edu).}
}

\begin{document}

\maketitle
\thispagestyle{empty}
\pagestyle{empty}

%%%%%%%%%%%%%%%%%%%%%%%%%%%%%%%%%%%%%%%%%%%%%%%%%%%%%%%%%%%%%%%%%%%%%%%%%%%%%%%%
\begin{abstract}
We generalize the Safe Extremum Seeking algorithm to address the minimization of an unknown objective function subject to multiple unknown inequality and equality constraints, relying on recent results of gradient flow systems. These constraints may represent safety or other critical conditions. The proposed ES algorithm functions as a general nonlinear programming tool, offering practical maintenance of all constraints and semiglobal practical asymptotic stability, utilizing a Lyapunov argument on the penalty function and the set-valued Lie derivative. The efficacy of the algorithm is demonstrated on a 2D problem.
\end{abstract}

%%%%%%%%%%%%%%%%%%%%%%%%%%%%%%%%%%%%%%%%%%%%%%%%%%%%%%%%%%%%%%%%%%%%%%%%%%%%%%%%
\section{Introduction}
Optimization under constraints is a central problem in control and engineering, appearing in applications ranging from robotics to autonomous systems and real-time decision-making. Various approaches exist for solving constrained optimization problems, but the choice of method depends significantly on the information available about the objective and constraint functions. In this paper, we leverage two recently proposed methods for addressing different problems using constrained optimization: one employing control barrier functions (CBFs) and quadratic programming (QP) for general nonlinear programming, under multiple constraints with known gradients, and another using extremum seeking (ES) to optimize an objective under a single CBF constraint, when gradients of the problem are unknown. These approaches tackle different problems, but their combination presents an opportunity to develop a more general ES tool merging the ideas from both algorithms.

The first work \cite{allibhoy2022control} introduces a general constrained optimization framework based on the safe gradient flow. The method treats a general nonlinear programming problem (under multiple inequality and equality constraints) as a control problem, where constraints are enforced through a CBF-inspired QP solved at every time instant. The result is a continuous-time optimization algorithm that ensures that the feasible set is forward invariant, while driving trajectories towards solutions to the nonlinear program. The approach has been shown~\cite{delimpaltadakis2024continuous} to be closely related to projected gradients methods, and may be understood as a continuous, smoothed-version of projected gradient descent methods. The work \cite{allibhoy2022control} assumes the gradients of the objective and constraints are known which may not apply to scenarios where only functions values are measured. 

In contrast, the work~\cite{williamssemiglobalsafety} presents a ``Safe'' Extremum Seeking (Safe ES) algorithm, designed for optimization problems where the objective and CBF gradients are unknown but can be estimated through real-time measurements. ES methods introduce perturbations to the optimization variables to estimate the gradient of an unknown function in real-time. To enforce safety constraints, Safe ES also incorporates a QP-based safety filter, which prevents constraint violations while seeking the extremum of the objective function. This approach is particularly useful in settings where analytical expressions of the constraint and objective are unavailable, and it ensures ``practical safety", meaning that the constraint violation can be made arbitrarily small via more conservative design constants. However, Safe ES, in its current form, is limited to a single inequality constraint as its design is based on the classical safety filter design in \cite{ames2016control}.

\textbf{Contributions:}  
We propose a Generalized Multi-Constraint Extremum Seeking (GMC-ES) algorithm that merges ES-based gradient estimation with a QP framework to handle multiple inequality and equality constraints using only function evaluations. The method solves a QP in real time with estimated gradients, and its aggressiveness in enforcing constraints is tunable through a design parameter. Using a penalty-function as a Lyapunov function, we establish semiglobal practical asymptotic stability for a class of nonconvex problems. More so, the practical maintenance of the feasible set is also achieved. This extends Safe ES to Generalized Multi-Constraint ES (GMC-ES), going beyond single-constraint settings, providing a real-time, constraint-enforcing optimizer for nonlinear programs. Along the way, we also provide conditions for global convergence of the safe gradient flow by means of the set-valued Lie derivative.

\textbf{Literature (ES, CBFs, safe gradient flow):}  ES is a powerful control method for dynamic systems with uncertainties, with its first formal stability proof established in the seminal work of \cite{krstic2000stability} paper, sparking renewed interest in this method. Constrained ES specifically has been explored in cases where constraints are unknown, often assuming convexity, incorporating techniques such as switching ES algorithms using projection maps \cite{chen2023continuoustime} and convex problem-solving via saddle points of modified barrier functions \cite{labar2019constrained}. Studies in \cite{durr2011smooth, durr2013saddle} provide Lie bracket-based analysis of saddle-point dynamics, while \cite{atta2019geometric} considers projection-based gradient methods.  
We recommend \cite{scheinker2024100} for a comprehensive survey on extremum seeking.

The literature on safe control has grown significantly, driven by foundational work in \cite{AmesAutomotive, Wieland}, particularly in the development of Control Barrier Functions (CBFs) and safety filters to enforce constraint satisfaction in real-time control applications \cite{ames2019control}.

The safe gradient flow algorithm \cite{allibhoy2022control} has found applications in power systems \cite{colot2024optimal}, reinforcement learning \cite{feng2023bridging,mestres2025anytime}, and more such as for feedback control design for dynamical systems \cite{chen2023online}. Here, we extend the range of applications to extremum seeking control. The connections of the safe gradient flow with other dynamical system-inspired approaches to solve constrained optimization problems, in particular those based on projected dynamical systems~\cite{nagurney2012projected,delimpaltadakis2024continuous}, are discussed in~\cite{allibhoy2022control}.

\section{Preliminaries on Nonsmooth Analysis}
Suppose $V: \mathbb{R}^n \rightarrow \mathbb{R}$ is locally Lipschitz, and $\Omega_V$ is a set of measure zero containing the points where $V$ is not differentiable. Denoting the convex hull of a set $\mathcal{S}$ as $\conv(\mathcal{S})$, the \textit{generalized gradient} is defined as
\begin{equation}
    \partial V (x) \coloneqq \conv \left\{ \lim_{i \to \infty} \nabla V(x_i) : x_i \to x, x_i \notin \Omega_V \right\} \,. \label{eqn:generalized_gradient} \nonumber \\
\end{equation}
Intuitively, the generalized gradient $\partial V (x)$, at some point $x$, describes a convex combination of all nearby gradients where $V(x)$ is strictly differentiable. See \cite{clarke1990optimization} and \cite{cortes2008discontinuous} for more information on generalized gradients. 

Regularity of a function proves useful in calculating generalized gradients. The definition of regularity can be seen in \cite{clarke1990optimization}, but the following class of regular functions is sufficient background for this paper, rewritten from \cite[Proposition 2.3.12]{clarke1990optimization}, which provides a useful expression for the generalized gradient.  
\begin{proposition} \label{prop:gen_grad}
    Let $g_k: \mathbb{R}^n \rightarrow \mathbb{R}$ be differentiable functions for $k = 1,2, \ldots,m$, and let $V: \mathbb{R}^n \rightarrow \mathbb{R}$ be defined as the maximum of the collection of functions 
    \begin{equation}
        V(x) = \max\{g_k(x): k=1,2,\ldots, m\} \,. \label{eqn:generic_V}
    \end{equation}
Let $I_V(x)$ denote the set of indices $k$ for which $V(x) = g_k(x)$ --- the ``active'' functions.
    Then:
\begin{enumerate}
    \item $V$ is regular and locally Lipschitz.
    \item The generalized gradient can be expressed as
\begin{equation}
    \partial V(x) = \conv \left\{ \nabla g_i(x) : i \in I_V(x) \right\} \,. \nonumber 
\end{equation}
\end{enumerate}
\end{proposition}

See \cite[Chapter 2]{clarke1990optimization} for more information on generalized gradients and useful calculus rules. Lipschitz functions we encounter in this work can often be written in the “pointwise maximum” form \eqref{eqn:generic_V} --- note these two basic identities
$
|g(x)|=\max \{ g(x), -g(x) \}
$
and 
$g(x) + [ h(x) ]_{+}=\max \{ g(x), g(x)+h(x) \}$ where,
%(where $ [h(x)]_{+}=\max\{h(x),0\}$). 
we denote $[h]_+ = \max \{0,h \}$ for $h \in \mathbb{R}$.

We now define a set-valued form of the Lie derivative, utilizing the generalized gradient, which is generally defined for a set-valued map but we give a simpler definition for a Lipschitz function, sufficient for this work \cite{cortes2008discontinuous}. 
Consider the system
\begin{equation}
    \dot x = f(x) \label{eqn:generic_ode}\,.
\end{equation}
\begin{definition} \label{def:set_valued_lie}
Let $V: \mathbb{R}^n \rightarrow \mathbb{R}$ and $f: \mathbb{R}^n \rightarrow \mathbb{R}^n$ be locally Lipschitz. The \textit{set-valued Lie derivative} of $V$ with respect to $f$ is 
\begin{equation}
    \mathcal{L}_f V(x) = \{ a \in \mathbb{R} : f(x)^\top p = a \text{ for all } p \in \partial V (x) \} \,. \nonumber 
\end{equation}
\end{definition}
If $V$ is differentiable, then one recovers the standard Lie derivative $\mathcal{L}_f V(x) = \{ f(x)^\top \nabla V(x) \}$. If $V$ if differentiable everywhere then with some abuse of notation we may ignore the set notation, writing $\mathcal{L}_f V(x) = f(x)^\top \nabla V(x) $.

The following result~\cite{bacciotti1999stability, cortes2008discontinuous} is a generalization of the invariance principle, and it is the primary tool used to study stability in this work. In the case where the set-valued quantity $\mathcal{L}_f V(x)$ is the empty set, we use the convention that $\max\{\emptyset \}=-\infty$.

\begin{theorem}[Invariance Principle] \label{thm:invariance_principle}
Let $V: \mathbb{R}^n \rightarrow \mathbb{R}$ be locally Lipschitz and regular. Let $\mathcal{S} \subset \mathbb{R}^n$ be compact and positively invariant for \eqref{eqn:generic_ode}. Suppose $\max {L}_f V(x) \leq 0$ for all $x \in \mathcal{S}$. Then all solutions of \eqref{eqn:generic_ode}, with $x(t_0) \in \mathcal{S}$, converge to the largest invariant set $\mathcal{M}$ contained in  
\begin{equation}
    \mathcal{S} \cap \overline{ \{x \in \mathbb R^n : 0 \in  {L}_f V(x)\} } \,. \nonumber 
\end{equation}
\end{theorem}

The next result is Nagumo's theorem, giving a useful condition which may be used to show that a set is invariant -- note that Theorem~\ref{thm:invariance_principle} \emph{requires} an invariant, compact set $\mathcal{S}$. The original work is given in \cite{nagumo1942lage}, but more simply reformulated from \cite[Theorem 4.7]{blanchini2008set}.

\begin{theorem}[Nagumo's Theorem] \label{thm:nagumo}
Suppose $\mathcal{S} \subseteq \mathbb{R}^n$ is a closed set. Then $\mathcal{S}$ is invariant with respect to \eqref{eqn:generic_ode} with $f$ locally Lipschitz if and only if
\begin{equation}
    f(x) \in \mathcal{T}_\mathcal{S}(x) \text{ for all } x \in \bd(\mathcal{S}) \, , \nonumber %
\end{equation}
where $\mathcal{T}_\mathcal{S}(x)$ is the Bouligand tangent cone \cite[Definition 4.6]{blanchini2008set}.
\end{theorem}

The following result characterizes Bouligand’s tangent cone for sublevel sets defined by the maximum of differentiable functions, showing that for a sublevel set $\mathcal{S} = \{ V(x) \leq \bar V \}$, with $V$ given in \eqref{eqn:generic_V}, the set $\mathcal{T}_\mathcal{S}(x)$ can be equivalently expressed using the set-valued Lie derivative $\mathcal{L}_f V(x)$

\begin{proposition} \label{prop:lie_deriv_tangent_cone}
    Let $f: \mathbb{R}^n \rightarrow \mathbb{R}^n$ be locally Lipschitz and $V: \mathbb{R}^n \rightarrow \mathbb{R}$ regular (which includes functions of the form as in \eqref{eqn:generic_V}). Consider $\mathcal{S} = \{x \in \mathbb{R}^n \, | \, V(x) \leq \bar V \}$ for $\bar V \in \mathbb{R}$. Then for $x \in \bd (\mathcal{S})$ such that $0 \notin \partial V(x)$,
    \begin{equation}
                \max \mathcal{L}_f V(x) \leq 0 \iff f(x) \in \mathcal{T}_\mathcal{S}(x) \,. \nonumber
    \end{equation}
\end{proposition}
\begin{proof}
    First, note that ``Bouligand's tangent cone'' in the terminology of \cite{clarke1990optimization} is referred to as the ``contingent cone'' --- see \cite[Chapter 7]{clarke1998nonsmooth}.  Since $V$ is regular by Proposition~\ref{prop:gen_grad} and $0 \notin \partial V(x)$ for $x \in \bd (\mathcal{S})$, then by \cite[Theorem 2.4.7]{clarke1990optimization} $S$ is regular and the ``tangent cone'' (in Clarke's terminology) is equivalent to Bouligand's tangent cone (or the contingent cone in Clarke's terminology) with $\mathcal{T}_S(x) = \{ v \in \mathbb{R}^n : V^\circ(x; v) \leq 0 \}$, where $V^\circ(x; v)$ is the ``generalized directional derivative'' of $V$ at the point $x$ in the direction $v$. By \cite[Proposition 2.1.2]{clarke1990optimization}, $V^\circ(x; v) = \max \{ \xi^\top v : \xi \in \partial V(x)\}$, and so 
    \begin{equation}
        \mathcal{T}_S(x) = \{ v \in \mathbb{R}^n : \max \{ \xi^\top v : \xi \in \partial V(x)\} \leq 0 \} \, . \nonumber
    \end{equation}

    Forward implication: take some $x \in \bd(S)$ and suppose $\max \mathcal{L}_f V(x) \leq 0 $. Then for all $a \in \mathcal{L}_f V(x)$, where $a = f(x)^\top p$ for all $p \in \partial V(x)$, we have $ a= f(x)^\top p \leq 0$, which implies $f(x) \in \mathcal{T}_S$.

    Backward implication: take some $x \in \bd(S)$ and suppose $f(x) \in \mathcal{T}_S$, implying $f(x)^\top p \leq 0$ for all $p \in \partial V(x)$. For any $a \in \mathcal{L}_f V(x)$ we have $a = f(x)^\top p$, for all $p \in \partial V(x)$. Since $a=f(x)^\top p \leq 0$ then $\max \mathcal{L}_f V(x) \leq 0$.       
\end{proof}

Theorem~\ref{thm:invariance_principle} directly implies that if the system $\dot x = f(x)$ has an equilibrium $x_e$ on some compact invariant set $\mathcal{S}$, and if $\max {L}_f V(x) < 0$ for $x \in \mathcal{S} \setminus \{x_e \}$, then solutions converge to $x_e$. But in the context of this work, we ultimately are interested in the convergence to a \emph{neighborhood} around a point. 

Consider instead the system
\begin{equation}
    \dot x = \tilde F(x) = F(x) + O(a) \, , \label{eqn:generic_ode_perturbed}
\end{equation}
where $F(x_e) = 0$ and $O(a)$ is Lipschitz in $x$ but can be made arbitrarily small by choosing the parameter $a>0$ sufficiently small (on a compact set). In this scenario, we would like to show that for any $\Delta>0$, trajectories of \eqref{eqn:generic_ode_perturbed} from the initial condition $x(0) \in \bar B_\Delta(x_e)$ converge to $\bar B_\nu(x_e)$, for any $\nu>0$, no matter how small, for some $0< a< a^*$. This is known as ``SPA stability'' or semiglobal practical asymptotic stability, see \cite[Assumption 2]{nesic2010unifying} for the definition of this property. Theorem~\ref{thm:invariance_principle} is the tool we use to demonstrate this SPA property in our work, and Theorem~\ref{thm:nagumo} is the tool used to find an invariant set $\mathcal{S}$, a requirement in Theorem~\ref{thm:invariance_principle}. 

\textbf{Organization:} In Section~\ref{sec:safe_grad_flow} we study the safe gradient flow algorithm and derive conditions for which global convergence is guaranteed. Rather than imposing conditions on the nonlinear program which is being solved, as studied in \cite{allibhoy2022control}, we extend the work and derive conditions on the QP defining the dynamics, and characterize results in terms of the set-valued Lie derivative. Then in Section~\ref{sec:extremum_seeking}, we give an extremum seeking design, GMC-ES, which mimics the safe gradient flow dynamics, extending the work in \cite{williamssemiglobalsafety} to handle multiple constraints. We also give a high level overview of the analysis and present a theorem showing practical convergence and practical anytime maintenance of all constraints. Finally in Section~\ref{sec:example} we implement GMC-ES in a simple 2D example. 

\section{Stability of the safe gradient flow}\label{sec:safe_grad_flow}
\subsection{Lipschitzness, Equilibrium, and Invariance}
In this section we state the form of the safe gradient flow from \cite{allibhoy2022control}, the basis for the extremum seeking design of this work, and characterize some key results in terms of the set-valued Lie derivative.
The results in this section differ slightly in nature to the safe gradient properties described in \cite{allibhoy2022control}. Here, we impose a \emph{global constraint qualification on the QP} \eqref{eqn:Ga_def} defining the safe gradient flow, whereas the work of \cite{allibhoy2022control} imposes \emph{constraint qualifications on the nonlinear program} \eqref{eqn:nonlinear_program}. We take the most direct route of imposing global conditions on the QP \eqref{eqn:Ga_def}, since our goal is to study global convergence of \eqref{eqn:exact_safe_grad_flow} under the dynamics \eqref{eqn:Ga_def}. Therefore, many of the results in this section, \emph{which hold globally}, have analogs in the work of \cite{allibhoy2022control}, \emph{which hold locally} or in a region near the feasible set $\mathcal{C}$. For example, Proposition~\ref{prop:equilibrium} has an analog \cite[Proposition 5.1]{allibhoy2022control} which holds given a constraint qualification on nonlinear program at an optimal point $\theta^*$. Proposition~\ref{prop:time_deriv_and_inv} also has an analog \cite[Theorem 5.4]{allibhoy2022control} given a constraint qualification on the nonlinear program for $\theta \in \mathcal{C}$.

Consider the nonlinear program:
\begin{equation} \label{eqn:nonlinear_program}
\begin{aligned} 
    & \underset{\theta \in \mathbb{R}^n}{\text{minimize}} && f(\theta) \\
    &\text{subject to} && g(\theta) \leq 0 \\
    & && h(\theta) = 0,
\end{aligned}
\end{equation}
where \( f : \mathbb{R}^n \rightarrow \mathbb{R} \), \( g : \mathbb{R}^n \rightarrow \mathbb{R}^m \), \( h : \mathbb{R}^n \rightarrow \mathbb{R}^l \) are differentiable with locally Lipschitz derivatives. Let
\begin{equation}
\mathcal{C} = \{ \theta \in \mathbb{R}^n \mid g(\theta) \leq 0, h(\theta) = 0 \},
\end{equation}
denote the feasible set.

Consider the safe gradient flow dynamics:
\begin{equation} \label{eqn:exact_safe_grad_flow}
    \dot \theta = \Ga(\theta)
\end{equation}
where $\Ga$ satisfies the quadratic program with $\alpha>0$:
\begin{align}
\Ga(\theta) = 
\underset{\xi \in \mathbb{R}^n}{\arg\min} \quad
& \left\{ \tfrac12 \|\xi + \nabla f(\theta)\|^2 \right\} \nonumber \\
\text{subject to} \quad
& \frac{\partial g(\theta)}{\partial \theta} \, \xi \leq - \alpha g(\theta), \label{eqn:Ga_def} \\
& \frac{\partial h(\theta)}{\partial \theta} \, \xi = - \alpha h(\theta) . \nonumber
\end{align}
The dynamics $\Ga$ can be thought of as a gradient descent based method, and under no constraints the minimization under $\xi$ results in $\xi^* = - \nabla f(\theta)$ --- standard gradient descent.

The Karush–Kuhn–Tucker (KKT) conditions for the optimal point $\xi^*$ along with associated multipliers $(u,v) \in \mathbb{R}^m \times \mathbb{R}^l$ are:
\begin{subequations} \label{eqn:kkt_qp_all}
\begin{align}
\xi^* + \nabla f(\theta) + \frac{\partial g(\theta)}{\partial \theta}^\top u + \frac{\partial h(\theta)}{\partial \theta}^\top v &= 0 \label{eqn:kkt_qp_a} \\
\frac{\partial g(\theta)}{\partial \theta} \, \xi^* +\alpha \, g(\theta) &\leq  0 \label{eqn:kkt_qp_b} \\
\frac{\partial h(\theta)}{\partial \theta} \, \xi^* + \alpha h(\theta) &=  0 \label{eqn:kkt_qp_c} \\
u &\geq 0 \label{eqn:kkt_qp_d} \\
u^\top \left( \frac{\partial g(\theta)}{\partial \theta} \, \xi^* + \alpha \, g(\theta) \right) &= 0 \label{eqn:kkt_qp_e} 
\end{align}
\end{subequations}
where $\xi^* = \Ga(\theta)$. Let the Lagrange multiplier set be defined as
\begin{align}\label{eqn:lag_mult_set}
    \Lambda(\theta) = &\{ (u,v) \in \mathbb{R}^m_{\geq 0} \times \mathbb{R}^k \; | \; \exists \xi^* \in \mathbb{R}^n \text{ such that} \nonumber \\ 
    & (\xi^*, u, v) \text{ solves \eqref{eqn:kkt_qp_all}} \} .
\end{align}

We impose the Mangasarian–Fromovitz constraint qualification (MFCQ) for all $\theta \in \mathbb{R}^n$ on the QP in \eqref{eqn:Ga_def}.

\begin{assumption}[MFCQ holds for the QP]\label{assum:MFCQ-QP}
MFCQ holds for the QP \eqref{eqn:Ga_def} at each $\theta \in \mathbb{R}^n$. Namely, for each \(\theta \in \mathbb{R}^n\):
\begin{enumerate}
    \item \(\frac{\partial h(\theta)}{\partial \theta}\) has full row rank, and
    \item \(\exists d\in\mathbb{R}^n\) such that
    \(\nabla h_j(\theta)^\top d = 0\) for all \(j=1,\dots,l\) and
    \(\nabla g_i(\theta)^\top d < 0\) for all \(i\in \mathcal{A}_{\mathrm{QP}}(\theta)\), where
\end{enumerate}
\begin{equation}  
\mathcal{A}_{\mathrm{QP}}(\theta):=\big\{ i\in\{1,\dots,m\}\; \big|\; \nabla g_i(\theta)^\top \xi^\ast + \alpha\, g_i(\theta)=0 \big\}. \nonumber
\end{equation}
\end{assumption}

\begin{proposition}[MFCQ implies bounded multipliers]
\label{prop:MFCQ-bounded}
Suppose Assumption~\ref{assum:MFCQ-QP} holds.
\begin{enumerate}[i)]
    \item Then the Lagrange multiplier set $\Lambda(\theta)$ associated with \eqref{eqn:kkt_qp_all} is nonempty and bounded. In particular, for all $\theta \in \mathbb{R}^n$ there exists an $M(\theta) < \infty$ such that every $(u,v) \in \Lambda(\theta)$ satisfies $u_i\le M(\theta)$ and $|v_j| \le M(\theta)$ for all $i \in \{1, \dots, m\}$ and $j \in \{ 1, \dots, l\}$. 
    \item Slater’s condition holds for \eqref{eqn:Ga_def} at the same $\theta$, which means $\exists \xi' \in\mathbb{R}^n$ satisfying $\nabla g_i(\theta)^\top \xi' + \alpha g_i(\theta)< 0$ and $\nabla h_j(\theta)^\top \xi' + \alpha h_j(\theta) =0$ for all $i \in \{1, \dots, m\}$ and $j \in \{ 1, \dots, l\}$.
\end{enumerate}
\end{proposition}
\begin{proof}
    Statement i) the set $\Lambda(\theta)$ is bounded by \cite{gauvin1977necessary}. Statement ii) see \cite[Proposition 3.3.8 proof]{bertsekas1997nonlinear}.
    \end{proof}

 The following proposition gives an additional fact and $\Ga(\theta)$ is locally Lipschitz.

\begin{proposition}[Lipschitzness of safe gradient flow] \label{prop:Ga_lipschitz}
Under Assumption~\ref{assum:MFCQ-QP}, for all $\theta \in \mathbb{R}^n$, $\Ga(\theta)$ is well-defined and locally Lipschitz.
\end{proposition}
\begin{proof}
    Under Assumption~\ref{assum:MFCQ-QP}, MFCQ provides a direction $d$ with $\nabla h_j^\top d=0$ and $\nabla g_i^\top d<0$ for all active $i$, so no inequality can be an equality for all feasible points; hence in \cite[(3.5)-(3.6)]{daniel1974continuity}, $B_0$ is vacuous, the strict–slack condition for $A_0$ follows from Slater's condition (implied by MFCQ), and the constant–rank hypothesis \cite[(3.8)]{daniel1974continuity} reduces to $\mathrm{rank}\,D(\theta)= l$ uniformly by the full row rank of $\partial h/ \partial \theta$. Since the the problem data in \eqref{eqn:Ga_def} are locally Lipschitz in $\theta$ and the quadratic cost is strongly convex, \cite[Theorem 3.10]{daniel1974continuity} applies, yielding $\Ga(\theta)$ locally Lipschitz in $\theta$.
\end{proof}

Under a constraint qualification, a local solution $\theta^* \in \mathcal{C}$ to \eqref{eqn:nonlinear_program} satisfies the KKT conditions
\begin{subequations} \label{eqn:kkt_np_all}
\begin{align} 
\nabla f(\theta^*) + \frac{\partial g(\theta^*)}{\partial \theta}^\top u^* + \ \frac{\partial h(\theta^*)}{\partial \theta}^\top v^* &= 0, \label{eqn:kkt_np_a}\\
g(\theta^*) &\leq 0, \label{eqn:kkt_np_b}\\
h(\theta^*) &= 0, \label{eqn:kkt_np_c}\\
u^* &\geq 0, \label{eqn:kkt_np_d}\\
(u^*)^\top g(\theta^*) &= 0, \label{eqn:kkt_np_e}
\end{align}
\end{subequations}
with some Lagrange multipliers $(u^*,v^*) \in \mathbb{R}^m \times \mathbb{R}^l$ associated with $\theta^* \in \mathbb{R}^n$ \cite{bertsekas1997nonlinear}. We make the assumption that the KKT conditions uniquely determines the global solution to the nonlinear problem \eqref{eqn:nonlinear_program}, in anticipation of a global study of convergence.
\begin{assumption} \label{assum:kkt}
    There exists a unique solution $\theta^* \in \mathbb{R}^n$ to \eqref{eqn:nonlinear_program}, and is the only point that satisfies \eqref{eqn:kkt_np_all} with some Lagrange multipliers $(u^*, v^*) \in \mathbb{R}^m \times \mathbb{R}^k$.
\end{assumption}

If $\theta^*$ is a local solution to \eqref{eqn:nonlinear_program}, and a constraint qualification holds at $\theta^*$, then the KKT conditions are satisfied at $\theta^*$. However, the converse is not true: a point $\theta$ may satisfy the KKT conditions, without being a (local or global) solution of \eqref{eqn:nonlinear_program}. Consequently, to ensure correctness in the study of the global convergence of \eqref{eqn:exact_safe_grad_flow}, we impose Assumption~\ref{assum:kkt}. The next result also follows from the work of \cite[Proposition 5.1]{allibhoy2022control}, demonstrating a single equilibrium of \eqref{eqn:exact_safe_grad_flow}.

\begin{proposition}[Equilibrium at KKT point] \label{prop:equilibrium} Under Assumptions~\ref{assum:MFCQ-QP} and \ref{assum:kkt}, $\Ga(\theta) = 0$ if and only if $\theta = \theta^*$. 
\end{proposition}
\begin{proof}
    If $\Ga(\theta)=\xi^*=0$, then the conditions in \eqref{eqn:kkt_qp_all} with $\alpha>0$ can be simplified to exactly that of \eqref{eqn:kkt_np_all}, which implies $\theta = \theta^*$ uniquely by Assumption~\ref{assum:kkt}. If $\theta = \theta^*$, notice that $\xi^*=0$ satisfies \eqref{eqn:kkt_qp_all} along with $u=u^*$ and $v=v^*$. Since the QP in \eqref{eqn:Ga_def} is strongly convex and Slater's condition holds, $\xi^*=0$ is unique.
\end{proof}

We define
\begin{equation}
\mathcal{C}_{\bar g, \bar h} = \{ \theta \in \mathbb{R}^n \mid g(\theta) \leq \bar g, |h(\theta)| \leq \bar h \},
\end{equation}
where $\bar g \in \mathbb{R}^m_{\geq 0}$, $\bar h \in \mathbb{R}^l_{\geq 0}$. Note that $\mathcal{C} \subseteq \mathcal{C}_{\bar g, \bar h}$.

\begin{proposition}[Time derivatives and Invariance] \label{prop:time_deriv_and_inv}
Under Assumptions~\ref{assum:MFCQ-QP} and \ref{assum:kkt} the following hold:
\begin{enumerate}[i)]
    \item All sets of the form $\mathcal{C}_{\bar g, \bar h}$, including the feasible set $\mathcal{C}$, are invariant with respect to \eqref{eqn:exact_safe_grad_flow}.
    \item For any $(u,v) \in \Lambda(\theta)$ satisfying \eqref{eqn:kkt_qp_all}, 
\begin{equation*}
    \mathcal{L}_{\Ga} f(\theta) =  -|| \Ga(\theta) ||^2 + \alpha u^\top g(\theta) + \alpha v^\top h(x)
\end{equation*}
\item For any $\theta \in \mathcal{C}$,
\begin{equation*}
    \mathcal{L}_{\Ga} f(\theta)  \leq 0
\end{equation*}
and $\mathcal{L}_{\Ga} f(\theta) = 0$ if and only if $\theta = \theta^*$.
\end{enumerate}    
\end{proposition}
\begin{proof}
    Statement i): $\xi^* = \Ga(\theta)$ are Lipschitz in $\theta$ and satisfy \eqref{eqn:kkt_qp_b} implying $\frac{d g_i(\theta)}{d t} = \nabla g_i(\theta)^\top \Ga(\theta) \leq -\alpha g_i(\theta)$ for all $i \in \{1, \dots, m\}$, according to the dynamics \eqref{eqn:exact_safe_grad_flow}.     
    Similarly, \eqref{eqn:kkt_qp_c} is satisfied, implying and $\frac{d h_j(\theta)}{d t} = - \alpha h_j(\theta)$ for $j \in \{1, \dots, l\}$. Therefore the sets $S_{g_i}=\{\theta: g_i(\theta)\leq \bar g_i \}$ and $S_{h_j}=\{\theta: |h_j(\theta)| \leq \bar h_j \}$ are invariant \cite[Lemma 3.4]{Khalil}, and since $\mathcal{C}_{\bar g, \bar h}$ is the intersection of all $S_{g_i}$ and $S_{h_j}$, then $\mathcal{C}_{\bar g, \bar h}$ is invariant. See \cite[Lemma 5.8]{allibhoy2022control} for proofs of statements ii) and iii).
\end{proof}

Proposition~\ref{prop:time_deriv_and_inv}i allows one to study trajectories on the invariant set $\mathcal{C}_{\bar g, \bar h}$, rather than all of $\mathbb{R}^n$. If $\mathcal{C}_{\bar g, \bar h}$ extends unbounded in some direction (it is not compact), then additional assumptions will need to be imposed for global stability, Theorem~\ref{thm:invariance_principle} requires a compact invariant set. Regardless, the invariance of $\mathcal{C}_{\bar g, \bar h}$ is useful in that it demonstrates that all $h_j(\theta(t))$ remain bounded above and below in time, and all $g_i(\theta(t))$ remain bounded above in time. 

\subsection{The Set-Valued Lie Derivative and Stability}
Consider the regular function $V_\epsilon: \mathbb{R}^n \rightarrow \mathbb{R}$ with parameter $\epsilon>0$,
\begin{equation} \label{eqn:Veps_def}
    V_\epsilon(\theta) = f(\theta) + \frac{1}{\epsilon} \sum_{i=1}^{m} [g_i(\theta)]_+ + \frac{1}{\epsilon} \sum_{j=1}^{l} |h_j(\theta)|.
\end{equation}
Then, using Proposition~\ref{prop:gen_grad}, and since $g(\theta)$ and $h(\theta)$ are differentiable, the generalized gradient is
\begin{align}
\partial V_\epsilon(\theta) =
& \biggl\{
\nabla f(\theta)
+ \frac{1}{\epsilon} \sum_{i=1}^m q_i \nabla g_i(\theta)
+ \frac{1}{\epsilon} \sum_{j=1}^l r_j \nabla h_j(\theta)
: \nonumber \\
& \;\;\; q \in Q(\theta),\;
r \in R(\theta)
 \biggr\}
\end{align}
where $q$ (associated with $g$) and $r$ (associated with $h$) are
\begin{equation}
Q(\theta) = \left\{
q \in \mathbb{R}^m :
q_i \in
\begin{cases}
\{0\} & \text{if } g_i(\theta) < 0 \\
\{1\} & \text{if } g_i(\theta) > 0 \\
{[0,c_1]} & \text{if } g_i(\theta) = 0 \quad
\end{cases}
\right\},
\end{equation}
\begin{equation}
R(\theta) = \left\{
r \in \mathbb{R}^l :
r_j \in
\begin{cases}
\{\sign(h_j(\theta))\} & \text{if } h_j(\theta) \neq 0 \\
{[-c_2,c_2]} & \text{if } h_j(\theta) = 0 \quad
\end{cases}
\right\},
\end{equation}
for some $c_1, c_2>0$, which are determined by the number of ``active'' functions (for $g_i = 0$ and $h_j = 0$) in order to satisfy the convex combination.

Therefore, the set-valued Lie derivative $\mathcal{L}_{\Ga} V_\epsilon(\theta)$ is
\begin{equation}
\begin{aligned}
& \mathcal{L}_{\Ga} V_\epsilon(\theta) = \bigg\{ \nabla f(\theta)^\top \Ga(\theta)
+ \frac{1}{\epsilon} \sum_{i=1}^m q_i \nabla g_i(\theta)^\top \Ga(\theta) \\
& \quad + \frac{1}{\epsilon} \sum_{j=1}^l r_j \nabla h_j(\theta)^\top \Ga(\theta)
: 
q \in Q(\theta),\;
r \in R(\theta)
\bigg\}.
\end{aligned}
\end{equation}
We take some element $p \in \mathcal{L}_{\Ga} V_\epsilon(\theta)$, and using Proposition~\ref{prop:time_deriv_and_inv}, we have $\mathcal{L}_{\Ga} f(\theta) = \nabla f(\theta)^\top \Ga(\theta)$ and
\begin{equation}
\begin{aligned}
& p = -\| \Ga(\theta) \|^2 + \alpha u^\top g(\theta) + \alpha v^\top h(\theta) \\
& \quad + \frac{1}{\epsilon} \sum_{i=1}^m q_i \nabla g_i(\theta)^\top \Ga(\theta) + \frac{1}{\epsilon} \sum_{j=1}^l r_j \nabla h_j(\theta)^\top \Ga(\theta) \, .
\end{aligned}
\end{equation}
From the rows of \eqref{eqn:kkt_qp_c}, we have
\begin{equation}
    r_j \nabla h_j(\theta)^\top \Ga(\theta) = \alpha r_j h_j(\theta) =-\alpha |h_j(\theta)|.
\end{equation}
and from the rows of \eqref{eqn:kkt_qp_b},
\begin{equation}
    q_i \nabla g_i(\theta)^\top \Ga(\theta) \leq -\alpha q_i g_i(\theta) = -\alpha [g_i(\theta)]_+.
\end{equation}
Therefore
\begin{equation}\label{eqn:bound_on_p_local1}
\begin{aligned}
& p \leq -\| \Ga(\theta) \|^2 +  \alpha \sum_{i=1}^m \left( u_i g_i(\theta) - \frac{[g_i(\theta)]_+}{\epsilon} \right) \\
& \quad  +  \alpha \sum_{j=1}^l \left( v_j h_j(\theta) - \frac{|h_j(\theta)|}{\epsilon}  \right) \, .
\end{aligned}
\end{equation}
Because $u_i g_i(\theta) \leq u_i [g_i(\theta)]_+$ and $v_j h_j(\theta) \leq |v_j| |h_j(\theta)|$ then
\begin{equation}\label{eqn:bound_on_p_local}
\begin{aligned}
& p \leq -\| \Ga(\theta) \|^2 +  \alpha \sum_{i=1}^m [g_i(\theta)]_+ \left( u_i - \frac{1}{\epsilon} \right) \\
& \quad  +  \alpha \sum_{j=1}^l | h_j(\theta)| \left( |v_j| - \frac{1}{\epsilon}  \right) \, .
\end{aligned}
\end{equation}
From this form of the bound on $p \in \mathcal{L}_{\Ga} V_\epsilon(\theta)$ we can see that if $\epsilon$ is chosen small, we can dominate the multiplier terms $u_i$ and $|v_j|$ in order to force $\max \mathcal{L}_{\Ga} V_\epsilon(\theta) < 0$. The constraint qualification, Assumption~\ref{assum:MFCQ-QP}, guarantees that $u, v$ are bounded. 

A simple condition which guarantees global asymptotic stability is that one of the inequality constraints $g_i(\theta)$ has compact levels -- or with many inequality constraints, the combination of sublevels yields a compact set. For example, consider $\theta \in \mathbb{R}^2$ with no equality constraints (so $\mathcal{C}_{\bar g, \bar h}$ reduces to $\mathcal{C}_{\bar g}$) and two inequality constraints. Suppose $\mathcal{C}_{\bar g} = \{\theta \in \mathbb{R}^2 \, | \, g_1(\theta) \leq \bar g \, , \, g_2(\theta) \leq \bar g \}$ for $\bar g > 0$ where $g_1(\theta) = -\theta_2$ and $g_2(\theta) =-1 + \theta_1^2 + \theta_2$. Therefore $g_i(\theta(t)) \leq g_i(\theta(0)) e^{-\alpha t}$ for $i=1,2$ and $\mathcal{C}_{\bar g}$ is compact and invariant for all $\bar g\geq 0$. 

\begin{assumption} \label{assum:C_bounded}
    For any $\theta$ there exists a compact $\mathcal{C}_{\bar g, \bar h}$ such that $\theta \in \mathcal{C}_{\bar g, \bar h}$.
\end{assumption}

Assumption~\ref{assum:C_bounded} implies any initial condition can be contained in a compact invariant set $\theta \in \mathcal{C}_{\bar g, \bar h}$, with invariance being guaranteed by Proposition~\ref{prop:time_deriv_and_inv}. Therefore Theorem~\ref{thm:invariance_principle} can directly be applied, with the Lyapunov function $V_\epsilon$ where $\epsilon>0$ is chosen sufficiently small to dominate all $u, v$, which are bounded on $\mathcal{C}_{\bar g, \bar h}$ by Proposition~\ref{prop:MFCQ-bounded}.

\begin{theorem} \label{thm:GAS_safe_grad}
    Let Assumptions \ref{assum:MFCQ-QP}, \ref{assum:C_bounded}, and \ref{assum:kkt} hold. Then the dynamics \eqref{eqn:exact_safe_grad_flow} are globally asymptotically stable with respect to the equilibrium $\theta^*$.
\end{theorem}
\begin{proof}
    Any initial condition $\theta(t_0)$ lies in a particular set $\mathcal{C}_{\bar g, \bar h}$ for some $\bar g, \bar h>0$, and $\mathcal{C}_{\bar g, \bar h}$ is compact (Assumption~\ref{assum:C_bounded}) and invariant (Proposition~\ref{prop:time_deriv_and_inv}). Consider \eqref{eqn:bound_on_p_local} for $p \in \mathcal{L}_{\Ga} V_\epsilon(\theta)$. Choose $\epsilon>0$ such that $\epsilon<1/\bar M$ where $\bar M = \max_{\theta \in \mathcal{C}_{\bar g, \bar h}} M(\theta)$ for $M(\theta)$ given in Proposition~\ref{prop:MFCQ-bounded}. Therefore $p < 0 $ for all $\theta \in \mathcal{C}_{\bar g, \bar h}\setminus\{ \theta^* \}$ and by Theorem~\ref{thm:invariance_principle} we have the result.
\end{proof}

\emph{Comments on when $\mathcal{C}_{\bar g, \bar h}$ is unbounded}:
Proving global stability on a potentially unbounded feasible set $\mathcal{C}$ (which means all sets $\mathcal{C}_{\bar g, \bar h}$ are unbounded) is more challenging. One will need to find an arbitrarily large compact invariant $\mathcal{M}$ (so that it subsumes any initial condition) without the use of $\mathcal{C}_{\bar g, \bar h}$. Instead, one could use of sublevel sets of the function $V_\epsilon$ (with radial unboundedness), since compactness of $\mathcal{C}_{\bar g, \bar h}$ cannot be relied upon. One challenge which must be dealt with is, for a fixed $\bar V>0$, making $\epsilon$ small (to dominate $u,v$ in \eqref{eqn:bound_on_p_local}) makes the sublevel $\mathcal{M}=\{\theta : V_\epsilon (\theta) \leq \bar V \}$ small, due to the form in \eqref{eqn:Veps_def}. In other words, the choice of a small $\epsilon$ is at odds with the goal of finding a large invariant set $\mathcal{M}$, as the compact invariant set is \textit{dependent} on $\epsilon$.

\section{Extremum Seeking} \label{sec:extremum_seeking}
\subsection{Design}
In the absence of knowledge of $\nabla
f, \nabla g_i,$ and $\nabla h_j$ we can consider an approximation of the design in \eqref{eqn:Ga_def}, through the use of techniques used in extremum seeking to estimate the gradients. The GMC-ES dynamics are given as
\begin{align}
\dot{\hat \theta} =& k \omegaf \hat{\mathcal{G}} _\alpha \label{eqn:th_dyn}\\
\dot G_f =& - \omegaf ( G_f - (f( \theta)-\eta_J) M(t) )        \label{eqn:gf_dyn}\\
\dot G_{g_i} =& -  \omegaf (G_{g_i} - (g_i( \theta) - \eta_{g_i})M(t) ) \label{eqn:gg_dyn} \\
\dot G_{h_j} =& -  \omegaf (G_{h_j} - (h_j( \theta) - \eta_{h_j})M(t) )  \label{eqn:gh_dyn} \\
\dot \eta_f =& -  \omegaf ( \eta_f -f( \theta )) \label{eqn:etaf_dyn}\\
\dot \eta_{g_i} =& -  \omegaf ( \eta_{g_i} -g_i( \theta )) \label{eqn:etag_dyn}\\
\dot \eta_{h_j} =& -  \omegaf ( \eta_{h_j} -h_j( \theta )) \label{eqn:etah_dyn}
\end{align}
for $i = 1, \dots, m$ and $j = 1, \dots, l$, where $\hat{\mathcal{G}}_\alpha$ represents the approximate solution to \eqref{eqn:Ga_def} and is defined as
the solution to the quadratic program with $\alpha>0$:
\begin{align}
\hat{\mathcal{G}}_\alpha = 
\underset{\xi \in \mathbb{R}^n}{\arg\min} \quad
& \left\{ \tfrac12 \|\xi + G_f\|^2 \right\} \nonumber \\
\text{subject to} \quad
& J_g \, \xi \leq - \alpha \eta_g, \label{eqn:approx_Ga_def} \\
& J_h \, \xi = - \alpha \eta_h . \nonumber
\end{align}
We define the relation
\begin{equation} \label{eqn:theta_theta_hat_S}
    \theta \coloneqq \hat \theta + S(t)
\end{equation}
for $i = 1, \cdots, m$ and $j = 1, \cdots, l$ where $G_f$, $G_{g_i}$, $G_{h_j} \in \mathbb{R}^n$ are estimates of $\nabla f(\theta)$, $\nabla g_i(\theta)$, $\nabla h_j(\theta)$, and 
$\eta_f, \eta_{g_i}, \eta_{h_j} \in \mathbb{R}$ are filtered measurements of $f(\theta), g_i(\theta), h_j(\theta)$ respectively. The Jacobian estimates of $g$ and $h$ are grouped as 
\begin{align}
    J_g &= 
    \begin{bmatrix}
    G_{g_1}^\top \\
    \vdots \\
     G_{g_m}^\top
    \end{bmatrix}, & \quad
    J_h &= 
    \begin{bmatrix}
     G_{h_1}^\top \\
    \vdots \\
    G_{h_k}^\top
    \end{bmatrix}.
\end{align}
for a design parameters $k, \omegaf, \alpha, \rho_\alpha, \epsilon_h >0$. Note that $\eta_g = [\eta_{g_1}, \cdots, \eta_{g_m}]^\top$ and $\eta_h = [\eta_{h_1}, \cdots, \eta_{h_k}]^\top$. As in classical extremum seeking, the perturbation signal $S$ and demodulation signal $M$ are given by $S_i(t) = a \sin(\omega_i t)$ and $M_i(t) = \frac{2}{a} \sin(\omega_i t) $
and contain additional design parameters $\omega_i, a \in \mathbb{R}_{>0}$. The parameter dynamics are a function of filtered measurements and gradients (shown in \eqref{eqn:approx_Ga_def} without arguments for space) of the nonlinear programming problem: $\hat{\mathcal{G}}_\alpha = \hat{\mathcal{G}}_\alpha(G_f, G_{g_1}, \dots, G_{h_1}, \dots, \eta_{f}, \dots, \eta_{g_1}, \dots, \eta_{h_1}, \dots)$.

\textbf{Naming Conventions:}  
following \cite{allibhoy2022control}, we use $\alpha$ for the constraint enforcement parameter, while \cite{williamssemiglobalsafety} used $c$. In \cite{williamssemiglobalsafety}, $h$ was used for inequality constraints, whereas here, $h$ are the equality constraints, and $g$ are the inequalities (where $g\leq0$ defines the feasible set, and not $g\geq 0$). In \cite{williamssemiglobalsafety}, $\alpha$ was used as a parameter in the Lyapunov analysis, but here the equivalent parameter is $\tfrac{1}{\epsilon}$, using conventions from \cite{allibhoy2022control}.

\subsection{Analysis}
In this section we give a high level preview of the analysis used to claim SPA stability of the system \eqref{eqn:th_dyn}-\eqref{eqn:etah_dyn}. 

We first perform the series of transformations \cite[Section 3]{nesic2010unifying} and summarize the details here:
\begin{enumerate}
    \item Transform \eqref{eqn:th_dyn}-\eqref{eqn:etah_dyn} to a new time variable $\tau = \omegaf t$.
    \item In the time scale $\tau$ perform averaging in the classical sense \cite{Khalil} over the common period of perturbation, deriving an average system.
    \item Make another time transformation to $s = k \tau$, and taking $k=0$ derive a singularly perturbed or ``reduced'' system and a boundary layer system, which is UGAS.
\end{enumerate}

We arrive at an averaged and then reduced system in the new time variable $s$:
\begin{equation}
    \frac{d \thetar}{d s} = \hat{\mathcal{G}}_{\alpha, \text{r}}(\thetar; a) \label{eqn:reduced_system}
\end{equation}
where 
\begin{align}
\hat{\mathcal{G}}_{\alpha, \text{r}}(\thetar; & a) = 
 \underset{\xi \in \mathbb{R}^n}{\arg\min} \quad
\left\{ \tfrac12 \|\xi + \nabla f(\theta_r) + O(a) \|^2 \right\} \nonumber \\
 \text{subj. to} \,
& \left( \frac{\partial g(\theta_r)}{\partial \theta_r} + O(a) \right)\, \xi \leq - \alpha \left( g(\theta_r) + O(a) \right), \label{eqn:red_avg_Ga_def} \\
& \left (\frac{\partial h(\theta_r)}{\partial \theta_r} + O(a) \right)\, \xi = - \alpha \left( h(\theta_r) + O(a) \right) . \nonumber
\end{align}

The terms $O(a)$ represent Lipschitz functions of $\thetar$ that have a norm which can be bounded linearly in $a>0$ --- so they can be made arbitrarily small by choosing a sufficiently small $a>0$. Since the problem data of the QP \eqref{eqn:red_avg_Ga_def} are perturbed by Lipschitz quantities, then by \cite[Theorem 3.10]{daniel1974continuity}, for sufficiently small $a$, we may express the solution $\hat{\mathcal{G}}_{\alpha, \text{r}}(\thetar; a)$ as
\begin{equation} \label{eqn:Ga_exact_plus_pert}
    \frac{d \thetar}{d s} = \hat{\mathcal{G}}_{\alpha, \text{r}}(\thetar; a) = \mathcal{G}_{\alpha}(\thetar) + O(a)
\end{equation}
where $\mathcal{G}_{\alpha}(\thetar)$ is the solution to the unperturbed QP \eqref{eqn:Ga_def}, where $a>0$ is taken sufficiently small. Therefore, the averaged and reduced system dynamics can be written as the exact safe gradient flow system with a small additive disturbance (proportional to the amplitude of the perturbation signals). We give the following without proof due to space limits.

\begin{theorem} [Stable and Practically Constrained] \label{thm:spa_stable_and_prac_constr}
Let Assumptions \ref{assum:MFCQ-QP},\ref{assum:kkt}, and \ref{assum:C_bounded} hold. Then there exists $\beta_\theta \in \mathcal{KL}$ such that: for any $\Delta, \delta, \nu>0$ there exist $r, \omegaf^*, a^*>0$, such that for any $\omegaf \in (0, \omegaf^*)$, $a \in (0,a^*)$, there exists $k^*(a)>0$ such that for any $k \in (0, k^*(a))$ the solutions to \eqref{eqn:th_dyn}-\eqref{eqn:etah_dyn} satisfy
\begin{align*}
    \| \hat\theta(t) - \theta^* \| &\leq \beta_\theta \left( \|\hat \theta (t_0) - \theta^*\|, k \cdot \omegaf \cdot (t-t_0) \right) + \nu, \label{eqn:theta_KL_bound} \\
    g_i(\theta(t))  &\leq g_i(\theta(t_0)) e^{- \alpha k \omegaf (t-t_0)} + O(\delta) \\
    \| h_j(\theta(t)) \| &\leq \| h_j(\theta(t_0)) \| e^{-\alpha k \omegaf (t-t_0)} + O(\delta)    
\end{align*}
for all $ \| \hat\theta(t_0) - \theta^* \| \leq \Delta$, $\| G_{g_i}(t_0) - \nabla g_i(\theta(t_0))\| \leq r$, $\| G_{h_j}(t_0) - \nabla h_j(\theta(t_0))\| \leq r$, $\| \eta_{g_i}(t_0) - g_i(\theta(t_0))\| \leq r$, $\| \eta_{h_j}(t_0) - h_j(\theta(t_0))\| \leq r$, for each $i=1, \dots, m$, $j=1, \dots, l$ and for all all $t \geq t_0\geq 0$.
\end{theorem}

Proof idea: under the assumptions discussed in Section~\ref{sec:safe_grad_flow}, the dynamics $\mathcal{G}_{\alpha}(\thetar)$ are locally Lipschitz, implying that the averaged and reduced dynamics of the ES scheme in \eqref{eqn:Ga_exact_plus_pert} are also locally Lipschitz for small $a>0$, so that the MFCQ assumption still holds for the QP \eqref{eqn:red_avg_Ga_def}. This is crucial as the work of \cite{nesic2010unifying} requires this system to be locally Lipschitz, admitting classical solutions (see the conditions in \cite[Lemma 1, Lemma 2]{tan2005non}, which are relied upon in the proof sketch of \cite[Theorem 1]{nesic2010unifying}). The result \cite[Theorem 1]{nesic2010unifying} states that if \eqref{eqn:Ga_exact_plus_pert} is SPA stable, then the GMC-ES scheme in \eqref{eqn:th_dyn}-\eqref{eqn:etah_dyn} will also be SPA stable. Theorem~\ref{thm:spa_stable_and_prac_constr} can be argued for using \cite[Theorem 1]{nesic2010unifying} and a similar Lyapunov based argument with the same $V_\epsilon$ which was used in Section~\ref{sec:safe_grad_flow} to study the exact system $\dot \theta= \mathcal{G}_\alpha(\theta)$. See \cite{williamssemiglobalsafety} as well, which provides a similar analysis albeit under a single inequality constraint (so $m = 1, l =0$). Finally, since the Lipschitzness of \eqref{eqn:approx_Ga_def} may not hold for the transient if the estimators in \eqref{eqn:gg_dyn}, \eqref{eqn:gh_dyn}, \eqref{eqn:etag_dyn} and \eqref{eqn:etah_dyn} are not close to the true values $\nabla g_i(\theta(t)), \nabla h_j(\theta(t)), g_i(\theta(t))$ and $h_j(\theta(t))$, we restrict the initialization of these estimators to be close to the their true values, so that \cite[Theorem 1]{nesic2010unifying} applies.

\begin{figure}[t]
\centering
  \includegraphics[width=.49\textwidth]{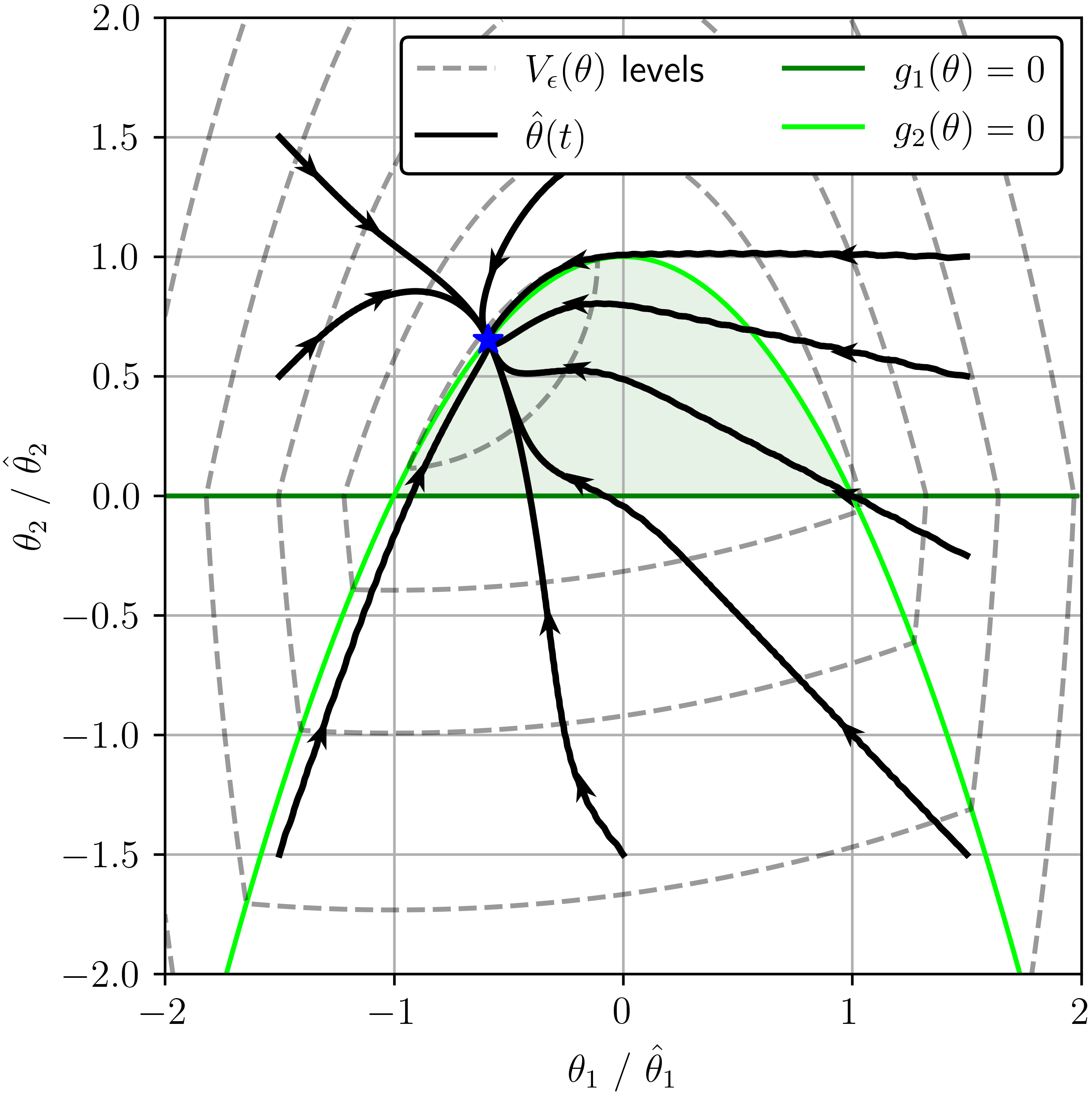}
\caption{Trajectories $\hat \theta(t)$ plotted from various initial conditions along with the zero level sets of $g_1, g_2$ and level sets of $V_\epsilon$ with $\epsilon = 0.1$. The inequality set $g(\theta) \leq 0$ is shaded in green, and the point $\theta^*$ is the blue star.}
\label{fig:example_state}
\end{figure}

\section{2D Example} \label{sec:example}

% \begin{figure}[t]
% \centering  \includegraphics[width=.49\textwidth]{images/time_plot_ACC.png}
% \caption{Plot of $\hat{\theta}(t)$ from $\hat{\theta}(t_0) = (1.5, -0.25)$ with gradient estimates compared with their true values at each time $t$.}
% \label{fig:example_time}
% \end{figure}
Consider the problem below with unknown functions 
\begin{align*}
    f(\theta) &= (\theta_1 + 1)^2 + (\theta_2-1)^2 \, , \\
    g_1(\theta) &= -\theta_2 \, , \\
    g_2(\theta) &= -1 + \theta_1^2 + \theta_2 \, .
\end{align*}
 The design constants are chosen as $a= 0.1 $, $\alpha=1 $, $k= 0.03$, $\omegaf= 0.5$, $\omega_1= 10$, and $\omega_2= 13$. The GMC-ES scheme is integrated with the Euler method and a $\Delta t = 0.048$. 
 %Fig.~\ref{fig:example_state} and Fig.~\ref{fig:example_time} demonstrate the performance of the algorithm in time and in parameter space. 
 Fig.~\ref{fig:example_state} demonstrates the performance of the algorithm in parameter space.

\emph{Assumption satisfaction:} Assumption~\ref{assum:MFCQ-QP}(ii) is nontrivial (no equalities). 
Since \(\nabla g_1(\theta)=(0,-1)^\top\) and \(\nabla g_2(\theta)=(2\theta_1,1)^\top\), when \(\theta_1\neq 0\) the two gradients are linearly independent, so we can find a single direction \(d\) with \(\nabla g_i(\theta)^\top d<0\) for both \(i=1,2\). When \(\theta_1=0\) the gradients are opposite, \((0,-1)\) and \((0,1)\). If both were active at the QP solution \(\xi^\ast\), the active equalities \(-\xi_2^\ast-\alpha \theta_2=0\) and \(\xi_2^\ast+\alpha(-1+\theta_2)=0\) implies \(-\alpha\theta_2=\alpha(1-\theta_2)\), a contradiction; hence \(|\mathcal A_{\mathrm{QP}}(\theta)|\le 1\). If the active set is \(\{i\}\), take \(d=-\nabla g_i(\theta)\) so that \(\nabla g_i^\top d=-\|\nabla g_i(\theta)\|^2<0\); if it is empty, the condition is vacuous. Therefore Assumption~\ref{assum:MFCQ-QP} holds for all \(\theta\).
Since $f$ is strongly convex and the feasible set is convex with Slater’s condition, \eqref{eqn:nonlinear_program} has a unique minimizer $\theta^*$, and  the KKT conditions are necessary and sufficient, so there exists at least one multiplier pair $(u^*,v^*)$ such that $(\theta^\ast,u^\ast,v^\ast)$ satisfies \eqref{eqn:kkt_np_all}, and no other $\theta$ does. Hence Assumption~\ref{assum:kkt} holds. One can also check that Assumption~\ref{assum:C_bounded} also holds.

\emph{Finding $\theta^*$:} The KKT conditions \eqref{eqn:kkt_np_all} state there exist $u_1,u_2\ge 0$ such that
\[
\begin{aligned}
&\text{(stationary)} && 2(\theta_1+1)+2u_2\theta_1=0,\\ & && 2(\theta_2-1)-u_1+u_2=0,\\
&\text{(feasibility)} && \theta_2\ge 0,\ \ \theta_1^2+\theta_2\le 1,\\
&\text{(complementary)} && u_1\,\theta_2=0,\ \ u_2\,(\theta_1^2+\theta_2-1)=0.
\end{aligned}
\]
Suppose $g_2$ is active at the solution: $\theta_1^2+\theta_2=1$ and $u_2\ge 0$.
If $\theta_2=0$ then $\theta_1=\pm1$, but neither choice satisfies the stationary condition with $u_1,u_2\ge 0$, so $\theta_2>0$ and thus $u_1=0$.
With $u_1=0$, the stationary condition gives $u_2=2(1-\theta_2)$ and, using $g_2=0$ (i.e., $\theta_2=1-\theta_1^2$), we obtain
\[
u_2=2\theta_1^2,\qquad (1+u_2)\theta_1+1=0 \ \Rightarrow\ 2\theta_1^3+\theta_1+1=0.
\]
This cubic has a unique real root and so $\theta_1^\ast=-0.58975,$
and $\theta_2^\ast=1-(\theta_1^\ast)^2 = 0.65219$.

\section{Conclusion}
We introduced the Generalized Multi-Constraint Extremum Seeking algorithm, extending Safe ES to handle multiple unknown inequality and equality constraints. Using the safe gradient flow with ES-based gradient estimates, we established semiglobal practical asymptotic stability and practical constraint satisfaction. This work broadens the use of ES to general nonlinear programming problems, providing a real-time, constraint-enforcing optimization tool.

\bibliographystyle{IEEEtranS}
%\bibliography{references}

\begin{thebibliography}{10}
\providecommand{\url}[1]{#1}
\csname url@samestyle\endcsname
\providecommand{\newblock}{\relax}
\providecommand{\bibinfo}[2]{#2}
\providecommand{\BIBentrySTDinterwordspacing}{\spaceskip=0pt\relax}
\providecommand{\BIBentryALTinterwordstretchfactor}{4}
\providecommand{\BIBentryALTinterwordspacing}{\spaceskip=\fontdimen2\font plus
\BIBentryALTinterwordstretchfactor\fontdimen3\font minus \fontdimen4\font\relax}
\providecommand{\BIBforeignlanguage}[2]{{%
\expandafter\ifx\csname l@#1\endcsname\relax
\typeout{** WARNING: IEEEtranS.bst: No hyphenation pattern has been}%
\typeout{** loaded for the language `#1'. Using the pattern for}%
\typeout{** the default language instead.}%
\else
\language=\csname l@#1\endcsname
\fi
#2}}
\providecommand{\BIBdecl}{\relax}
\BIBdecl

\bibitem{allibhoy2022control}
A.~Allibhoy and J.~Cort{\'e}s, ``Control barrier function-based design of gradient flows for constrained nonlinear programming,'' \emph{IEEE Transactions on Automatic Control}, vol.~69, no.~6, pp. 3499--3514, 2024.

\bibitem{AmesAutomotive}
A.~D. Ames, X.~Xu, J.~W. Grizzle, and P.~Tabuada, ``Control barrier function based quadratic programs for safety critical systems,'' \emph{IEEE Transactions on Automatic Control}, vol.~62, pp. 3861--3876, 2017.

\bibitem{ames2019control}
A.~D. Ames, S.~Coogan, M.~Egerstedt, G.~Notomista, K.~Sreenath, and P.~Tabuada, ``Control barrier functions: Theory and applications,'' in \emph{2019 18th European control conference (ECC)}.\hskip 1em plus 0.5em minus 0.4em\relax Ieee, 2019, pp. 3420--3431.

\bibitem{ames2016control}
A.~D. Ames, X.~Xu, J.~W. Grizzle, and P.~Tabuada, ``Control barrier function based quadratic programs for safety critical systems,'' \emph{IEEE Transactions on Automatic Control}, vol.~62, no.~8, pp. 3861--3876, 2016.

\bibitem{atta2019geometric}
K.~T. Atta, M.~Guay, and R.~Lucchese, ``A geometric phasor extremum seeking control approach with measured constraints,'' in \emph{2019 IEEE 58th Conference on Decision and Control (CDC)}.\hskip 1em plus 0.5em minus 0.4em\relax IEEE, 2019, pp. 1494--1500.

\bibitem{bacciotti1999stability}
A.~Bacciotti and F.~Ceragioli, ``Stability and stabilization of discontinuous systems and nonsmooth lyapunov functions,'' \emph{ESAIM: Control, Optimisation and Calculus of Variations}, vol.~4, pp. 361--376, 1999.

\bibitem{bertsekas1997nonlinear}
D.~P. Bertsekas, ``Nonlinear programming,'' \emph{Journal of the Operational Research Society}, vol.~48, no.~3, pp. 334--334, 1997.

\bibitem{blanchini2008set}
F.~Blanchini and S.~Miani, \emph{Set-Theoretic Methods in Control}.\hskip 1em plus 0.5em minus 0.4em\relax Springer, 2008, vol.~78.

\bibitem{chen2023continuoustime}
X.~Chen, J.~I. Poveda, and N.~Li, ``Continuous-time zeroth-order dynamics with projection maps: Model-free feedback optimization with safety guarantees,'' \emph{arXiv preprint arXiv:2303.06858}, 2023.

\bibitem{chen2023online}
Y.~Chen, L.~Cothren, J.~Cort{\'e}s, and E.~Dall’Anese, ``Online regulation of dynamical systems to solutions of constrained optimization problems,'' \emph{IEEE Control Systems Letters}, vol.~7, pp. 3789--3794, 2023.

\bibitem{clarke1990optimization}
F.~H. Clarke, \emph{Optimization and Nonsmooth Analysis}.\hskip 1em plus 0.5em minus 0.4em\relax SIAM, 1990.

\bibitem{clarke1998nonsmooth}
F.~H. Clarke, Y.~S. Ledyaev, R.~J. Stern, and R.~Wolenski, \emph{Nonsmooth analysis and control theory}.\hskip 1em plus 0.5em minus 0.4em\relax Springer, 1998.

\bibitem{colot2024optimal}
A.~Colot, Y.~Chen, B.~Corn{\'e}lusse, J.~Cort{\'e}s, and E.~Dall’Anese, ``Optimal power flow pursuit via feedback-based safe gradient flow,'' \emph{IEEE Transactions on Control Systems Technology}, vol.~33, pp. 658--670, 2025.

\bibitem{cortes2008discontinuous}
J.~Cort\'es, ``Discontinuous dynamical systems,'' \emph{IEEE Control Systems Magazine}, vol.~28, no.~3, pp. 36--73, 2008.

\bibitem{daniel1974continuity}
J.~W. Daniel, ``The continuity of metric projections as functions of the data,'' \emph{Journal of Approximation Theory}, vol.~12, no.~3, pp. 234--239, 1974.

\bibitem{delimpaltadakis2024continuous}
G.~Delimpaltadakis, J.~Cortés, and W.~P. M.~H. Heemels, ``{Continuous approximations of projected dynamical systems via Control Barrier Functions},'' \emph{IEEE Transactions on Automatic Control}, vol.~70, no.~1, pp. 681--688, 2024.

\bibitem{durr2011smooth}
H.-B. D{\"u}rr and C.~Ebenbauer, ``A smooth vector field for saddle point problems,'' in \emph{2011 50th IEEE Conference on Decision and Control and European Control Conference}.\hskip 1em plus 0.5em minus 0.4em\relax IEEE, 2011, pp. 4654--4660.

\bibitem{durr2013saddle}
H.-B. D{\"u}rr, C.~Zeng, and C.~Ebenbauer, ``Saddle point seeking for convex optimization problems,'' \emph{IFAC Proceedings Volumes}, vol.~46, no.~23, pp. 540--545, 2013.

\bibitem{feng2023bridging}
J.~Feng, W.~Cui, J.~Cort{\'e}s, and Y.~Shi, ``Bridging transient and steady-state performance in voltage control: A reinforcement learning approach with safe gradient flow,'' \emph{IEEE Control Systems Letters}, vol.~7, pp. 2845--2850, 2023.

\bibitem{gauvin1977necessary}
J.~Gauvin, ``A necessary and sufficient regularity condition to have bounded multipliers in nonconvex programming,'' \emph{Mathematical Programming}, vol.~12, no.~1, pp. 136--138, 1977.

\bibitem{Khalil}
H.~K. Khalil, \emph{{Nonlinear Systems}}, 3rd~ed.\hskip 1em plus 0.5em minus 0.4em\relax Upper Saddle River, NJ: Prentice-Hall, 2002.

\bibitem{krstic2000stability}
M.~Krstic and H.-H. Wang, ``Stability of extremum seeking feedback for general nonlinear dynamic systems,'' \emph{Automatica}, vol.~36, no.~4, pp. 595--601, 2000.

\bibitem{labar2019constrained}
C.~Labar, E.~Garone, M.~Kinnaert, and C.~Ebenbauer, ``Constrained extremum seeking: a modified-barrier function approach,'' \emph{IFAC-PapersOnLine}, vol.~52, no.~16, pp. 694--699, 2019.

\bibitem{mestres2025anytime}
P.~Mestres, A.~Marzabal, and J.~Cortes, ``Anytime safe reinforcement learning,'' in \emph{7th Annual Learning for Dynamics$\backslash$\& Control Conference}.\hskip 1em plus 0.5em minus 0.4em\relax PMLR, 2025, pp. 221--232.

\bibitem{nagumo1942lage}
M.~Nagumo, ``{\"U}ber die lage der integralkurven gew{\"o}hnlicher differentialgleichungen,'' \emph{Proceedings of the Physico-Mathematical Society of Japan. 3rd Series}, vol.~24, pp. 551--559, 1942.

\bibitem{nagurney2012projected}
A.~Nagurney and D.~Zhang, \emph{Projected dynamical systems and variational inequalities with applications}.\hskip 1em plus 0.5em minus 0.4em\relax Springer Science \& Business Media, 2012, vol.~2.

\bibitem{nesic2010unifying}
D.~Nesi{\'c}, Y.~Tan, W.~H. Moase, and C.~Manzie, ``A unifying approach to extremum seeking: Adaptive schemes based on estimation of derivatives,'' in \emph{49th IEEE conference on Decision and Control (CDC)}.\hskip 1em plus 0.5em minus 0.4em\relax IEEE, 2010, pp. 4625--4630.

\bibitem{scheinker2024100}
A.~Scheinker, ``100 years of extremum seeking: A survey,'' \emph{Automatica}, vol. 161, p. 111481, 2024.

\bibitem{tan2005non}
Y.~Tan, D.~Ne{\v{s}}i{\'c}, and I.~Mareels, ``On non-local stability properties of extremum seeking control,'' \emph{IFAC Proceedings Volumes}, vol.~38, no.~1, pp. 550--555, 2005.

\bibitem{Wieland}
P.~Wieland and F.~Allg{\"{o}}wer, ``Constructive safety using control barrier functions,'' \emph{IFAC Proceedings Volumes}, vol.~40, no.~12, pp. 462 -- 467, 2007, 7th IFAC Symposium on Nonlinear Control Systems.

\bibitem{williamssemiglobalsafety}
A.~Williams, M.~Krstic, and A.~Scheinker, ``Semiglobal safety-filtered extremum seeking with unknown {CBFs},'' \emph{IEEE Transactions on Automatic Control}, vol.~70, no.~3, pp. 1698--1713, 2025.

\end{thebibliography}
% Generated by IEEEtranS.bst, version: 1.14 (2015/08/26)

\end{document}